\documentclass[12pt,a4paper]{amsart}
\usepackage[bottom]{footmisc}
\usepackage{mathrsfs}
\usepackage{amsmath}
\usepackage{amsthm}
\usepackage{amsfonts}
\usepackage{latexsym}
\usepackage{latexsym}
\usepackage{graphicx}
\usepackage{hyperref}
\usepackage{amssymb}
\usepackage{subfigure}
\usepackage{epsf}
\usepackage{hyperref,amsmath,amsfonts,amsthm}
\usepackage{float}
\allowdisplaybreaks \makeatletter
\def\rightharpoonfill@{\arrowfill@\relbar\relbar\rightharpoonup}
\DeclareRobustCommand{\overrightharpoon}{\mathpalette{\underarrow@\rightharpoonfill@}}
\makeatother

\begin{document}
\newcommand{\beq}{\begin{equation}}
\newcommand{\eneq}{\end{equation}}
\newtheorem{thm}{Theorem}[section]
\newtheorem{coro}[thm]{Corollary}
\newtheorem{lem}[thm]{Lemma}
\newtheorem{prop}[thm]{Proposition}
\newtheorem{defi}[thm]{Definition}
\newtheorem{rem}[thm]{Remark}
\newtheorem{cl}[thm]{Claim}
\title{Canonical duality approach in the approximation of optimal Monge mass transfer mapping}
\author{Yanhua Wu$^1$\ \ \ \ \ \ Xiaojun Lu$^{2}$}
\thanks{Corresponding author: Xiaojun Lu}
\thanks{Email addresses: yhwu85@126.com(Yanhua Wu), lvxiaojun1119@hotmail.de(Xiaojun Lu)}
\date{}
\maketitle
\begin{center}
1. Department of Sociology, School of Public Administration, Hohai
University, 211189, Nanjing, China\\
2. Department of Mathematics \& Jiangsu Key Laboratory of
Engineering Mechanics, Southeast University, 210096, Nanjing,
China\\
\end{center}
\date{}
\maketitle
\begin{abstract} This paper mainly addresses the Monge mass transfer
problem in the 1-D case. Through an ingenious approximation
mechanism, one transforms the Monge problem into a sequence of
minimization problems, which can be converted into a sequence of
nonlinear differential equations with constraints by variational
method. The existence and uniqueness of the solution for each
equation can be demonstrated by applying the canonical duality
method. Moreover, the duality method gives a sequence of perfect
dual maximization problems. In the final analysis, one constructs
the approximation of optimal mapping for the Monge problem according
to the theoretical results.
\end{abstract}
\renewcommand{\abstractname}{R\'{e}sum\'{e}}
\begin{abstract}
Dans cet article, on consid\`{e}re essentiellement le probl\`{e}me
des transports des masses de Monge en dimension 1. En utilisant un
m\'{e}canisme d'approximation ing\'{e}nieuse, on transforme le
probl\`{e}me de Monge en une s\'{e}quence de probl\`{e}mes de
minimisation, qui peut \^{e}tre convertie en une s\'{e}quence des
\'{e}quations diff\'{e}rentielles nonlin\'{e}aires avec des
contraintes par la m\'{e}thode variationelle. En particulier, nous
prouvons l'existence et l'unicit\'{e} de la solution en appliquant
la m\'{e}thode de dualit\'{e} canonique. De plus, la transformation
de dualit\'{e} donne une s\'{e}quence des duals parfaits de
maximisation. Enfin, nous \'{e}tudions le probl\`{e}me de Monge
selon les r\'{e}sultats th\'{e}oriques.
\end{abstract}

\section{Introduction}
The Monge mass transfer model is widely used in modern social and
economic activities, medical science and mechanical processes, etc.
In these respects, some typical examples include the migration
problem, distribution of industrial products, purification of blood
in the kidneys and livers, shape optimization, etc. Interested
readers can
refer to \cite{LA1,LA2,Evans1,K1,K2,Monge,Su} for more details.\\

The original transfer problem, which was proposed by Monge
\cite{Monge}, investigated how to move one mass distribution to
another one with the least amount of work. In this paper, we
consider the Monge problem in the 1-D case. Let $\Omega=[a,b]$ and
$\Omega^*=[c,d]$, $a,b,c,d\in\mathbb{R}$. Here we focus on the
closed case, and other bounded cases can be discussed similarly.
Moreover, $f^+$ and $f^-$ are two nonnegative density functions in
$\Omega$ and $\Omega^*$, respectively, and satisfy the normalized
balance condition
$$\int_\Omega f^+dx=\int_{\Omega^*}f^-dx=1.$$ Let $c:\Omega\times
\Omega^*\to[0,+\infty)$ be a cost function, which indicates the work
required to move a unit mass from the position $x$ to a new position
$y$. There are many types of cost functions while dealing with
different problems \cite{LA1,Ca1,Evans1,Wang1}. In the Monge
problem, the cost function is proportional to the distance $|x-y|$,
for simplicity, $c(x,y)=|x-y|.$ The Monge problem consists in
finding an optimal mass transfer mapping ${\bf s}^*:\Omega\to
\Omega^*$ to minimize the cost functional $I({\bf s})$: \beq I({\bf
s}^*)=\displaystyle\min_{{\bf s}\in\mathscr{N}}\Big\{I[{\bf
s}]:=\int_{\Omega}|x-{\bf s}(x)|f^+(x)dx\Big\},\eneq where ${\bf
s}:\Omega\to \Omega^*$ belongs to the class $\mathscr{N}$ of
measurable one-to-one mappings driving $f^+(x)$ to $f^-(y)$.\\

In the 1940s, Kantorovich initiated a duality theory by relaxing
Monge transfer problem to the task of finding a maximizer for the
Kantorovich problem \cite{K1,K2}. This mechanism plays an archetypal
role in the infinite-dimensional linear programming \cite{V}. As a
matter of fact, the Kantorovich problem may not be a perfect dual to
the Monge problem unless a so-called {\it dual criteria for
optimality} is satisfied \cite{Ca1,Evans1}. Indeed, a huge body of
mathematical tools have been developed for computing the maximizer,
such as the Monge-Kantorovich-Rubinstein-Wasserstein matrices
\cite{R}, dual potentials for capacity constrained optimal transport
\cite{Mc1}, etc.\\

It turns out that the nonuniform convexity of the cost function
$c(x,y)$ defeats many simple attempts to sort out the structure of
optimal mass allocation. In order to gain some insight into this
problem, many mathematicians introduced lots of approximating
mechanisms. For example, L. A. Caffarelli, W. Gangbo, R. J. McCann
and X. J. Wang \cite{Ca2,GC1,GC2,Wang2}, etc. utilized an
approximation of strictly convex cost functions
$$c_\epsilon(x,y)=|x-y|^{1+\epsilon}\ \ \ \epsilon>0.$$ The existence and uniqueness
of the optimal mapping ${\bf s}^*_\epsilon$ can be proved by convex
analysis. Then let $\epsilon$ tends to $0$, and one can construct an
optimal mapping ${\bf s}^*$ by using transfer rays and transfer sets
invoked by L. C. Evans and W. Gangbo \cite{Evans2}. In addition, N.
S. Trudinger and X. J. Wang etc. used the approximation
$$c_{\epsilon}(x,y)=\sqrt{\epsilon^2+|x-y|^2}$$
in the discussion of regularity \cite{Wang1,Wang2}. Moreover, L. C.
Evans, W. Gangbo and J. Moser \cite{DM,Evans1,Evans2} provided an
ODE recipe to build ${\bf s}^*$ by solving a flow problem involving
$Du$. This
method is very useful but really complicated.\\

In this paper, we consider the approximation of an optimal mapping
through solving the optimization of distribution density in the
probability theory. Here, we mainly consider two typical cases of
$\Omega\bigcap\Omega^*=\emptyset$, namely,
\begin{itemize}
\item {\bf Assumption I}: $0\leq c<d<a<b$, $a>>0$, $d-c$ sufficiently large, specified in Lemma 2.9;
\item {\bf Assumption II}: $a<b<c<d\leq0$, $b<<0$, $d-c$ sufficiently large, specified in Lemma 2.9.
\end{itemize}
Let $\alpha>0$ be sufficiently large and consider the distribution
densities subject to \beq u\in W_0^{1,\infty}(\Omega^*)\cap
C(\overline{\Omega^*}), \eneq \beq u\geq 0,\ \text{a.e.\ in}\
\Omega^*,\eneq
\beq\|u\|_{L^1(\Omega^*)}=1,\eneq
\beq\|u_y\|_{L^\infty(\Omega^*)}\leq\alpha,\eneq where
$W_0^{1,\infty}(\Omega^*)$ is the Sobolev space and $u_y$ stands for
the weak derivative with respect to $y\in\Omega^*$.
It is evident that the $\delta$-function is excluded from our
discussion.
On the one hand, under the Assumption I, Monge transfer problem (1)
can be converted into a maximization of the expectation of the
real-valued random variable $Y\in\Omega^*$ with respect to the
distribution densities $u$ subject to (2)-(5), \beq (\mathcal
{P}^{(1)}):\displaystyle\max_{{u}}\Big\{\mathbb{E}_{u}(Y)
:=\int_{\Omega^*}yu(y)dy\Big\},\eneq with
\begin{itemize}
\item {\bf Assumption III}: The optimal mapping ${\bf s^*}$ is strictly
monotonous. Under Assumption I, ${\bf s^*}(a)=d$ or ${\bf
s^*}(b)=d$; while under Assumption II, ${\bf s^*}(a)=c$ or ${\bf
s^*}(b)=c$.
\end{itemize}
On the other hand, under the Assumptions II and III, Monge transfer
problem (1) can be converted into a minimization of the expectation
of the real-valued random variable $Y\in\Omega^*$ with respect to
the distribution densities $u$ subject to (2)-(5),, \beq (\mathcal
{P}^{(2)}): \displaystyle\min_{{u}}\Big\{\mathbb{E}_{u}(Y)
:=\int_{\Omega^*}yu(y)dy\Big\}.\eneq

In this paper, we investigate the analytic approximating mapping
through {\it canonical duality method} introduced by David. Y. Gao
and G. Strang \cite{G1,G2,G3}. This theory was originally proposed
to find minimizers for a non-convex strain energy functional with a
double-well potential. During the last few years, considerable
effort has been taken to illustrate these non-convex problems from
the theoretical point of view. Through applying this method, David
Y. Gao and G. Strang characterized the local energy extrema and the
global energy minimizer for both hard device and soft device and
finally obtained the analytical solutions. Readers can refer to
\cite{G4,G5,G6,G7}.\\

Inspired by the survey paper \cite{Evans3}, we propose a nonlinear
differential equation approach by introducing a sequence of
approximation problems of the primal $(\mathcal {P}^{(1)})$ and
$(\mathcal {P}^{(2)})$, namely,
\begin{equation}(\mathcal{P}^{(\varepsilon)}):
\displaystyle\min_{w_\varepsilon}\Big\{I^{(\varepsilon)}[w_\varepsilon]:=\int_{\Omega^*}
L^{(\varepsilon)}(w_{\varepsilon,y},w_\varepsilon,y)dy:=\int_{\Omega^*}
\Big(H^{(\varepsilon)}(w_{\varepsilon,y})-w_\varepsilon
|y|\Big)dy\Big\},
\end{equation}
where $H^{(\varepsilon)}:\mathbb{R}\to\mathbb{R}^+$ is defined as
\[
H^{(\varepsilon)}(\gamma):=\varepsilon{\rm
e}^{(\gamma^2-\alpha^2)/(2\varepsilon)}.
\]
Moreover,
$$L^{(\varepsilon)}(P,z,y):\mathbb{R}\times\mathbb{R}\times\Omega^*\to
\mathbb{R}$$ satisfies the following coercivity inequality and is
convex in the variable $P$, \beq L^{(\varepsilon)}(P,z,y)\geq
p_{\varepsilon}|P|^2-q_\varepsilon,\ P,z\in\mathbb{R}, y\in\Omega^*,
\eneq for constants $p_\varepsilon$ and $q_\varepsilon$.
$I^{(\varepsilon)}$ is called the potential energy functional and is
weakly lower semicontinuous on $W^{1,\infty}_0(\Omega^*)$. It's
worth noticing that when $|\gamma|\leq\alpha$, then
$$\displaystyle\lim_{\varepsilon\to
0^+}H^{(\varepsilon)}(\gamma)=0$$ uniformly. From \cite{Evans4}, one
knows immediately there exists a distribution density
$\bar{u}_\varepsilon$ solving
$$I^{(\varepsilon)}[\bar{u}_\varepsilon]=\displaystyle\min_{w_\varepsilon}\Big\{I^{(\varepsilon)}[w_\varepsilon]\Big\}.$$
Consequently, once such a sequence of functions
$\{\bar{u}_\varepsilon\}_\varepsilon$ is obtained, then it will help
find an optimal distribution density which solves the primal
problems $(\mathcal {P}^{(1)})$ or $(\mathcal {P}^{(2)})$. This
paper is aimed to obtain an explicit representation of this
approximation sequence. Generally speaking, there are plenty of
approximating schemes, for example, one can also let
$$H^{(\varepsilon)}(\gamma):=\varepsilon(\gamma^2-\alpha^2)^2.$$ Then by following the
procedure in dealing with double-well potentials in \cite{G1,G7}, we
could definitely find an optimal distribution density.

By variational calculus, correspondingly, one derives a sequence of
Euler-Lagrange equations for $(\mathcal{P}^{(\varepsilon)})$, \beq
\begin{array}{ll}\displaystyle ({\rm
e}^{(u_{\varepsilon,y}^2-\alpha^2)/(2\varepsilon)}u_{\varepsilon,y})_y+|y|=0,&
\ \text{\rm in}\ U^{(\varepsilon)},
\end{array}\eneq
equipped with the Dirichlet boundary condition, where the compact
support
$$U^{(\varepsilon)}:=\text{Supp}(u_\varepsilon)\subset\Omega^*$$ is connected and will
be determined later. The term ${\rm
e}^{(u_{\varepsilon,y}^2-\alpha^2)/(2\varepsilon)}$ is called the
transport density. As a matter of fact,
$\{u_\varepsilon\}_\varepsilon$ is a sequence of strictly concave
functions. Clearly, like $p-$Laplacian, ${\rm
e}^{(u_{\varepsilon,y}^2-\alpha^2)/(2\varepsilon)}$ is a highly
nonlinear function, which is difficult to solve by the direct
approach \cite{JH,Evans4,LIONS}. However, by the canonical duality
theory, one is able to demonstrate the existence and uniqueness of
the solution of the Euler-Lagrange equation, which establishes the
equivalence between the local minimizer of
($\mathcal{P}^{(\varepsilon)}$)
and the solution of Euler-Lagrange equation (5). This will help find a global minimizer of ($\mathcal{P}^{(\varepsilon)}$).\\

At the moment, we would like to introduce the main theorems. First,
we consider the approximation problem of (8).
\begin{thm}
For any $\varepsilon>0$, there exists a sequence of solutions
$\{\bar{u}_\varepsilon\}_\varepsilon$ satisfying (2)-(5) for the
Euler-Lagrange equations (10), which is at the same time a sequence
of global minimizers for the approximation problems
($\mathcal{P}^{(\varepsilon)}$) in the following form,
\begin{itemize}
\item Under Assumption I,

\[
\bar{u}_\varepsilon(y)= \left\{
\begin{array}{cl}
\displaystyle\int_{d}^{y}(-G(t)+C_\varepsilon(d))/E_\varepsilon^{-1}((-G(t)+C_\varepsilon(d))^2)dt,&
y\in[p_\varepsilon^*(d),d]\subset\Omega^*,\\
\\
0,& \text{elsewhere}\ \text{in}\ \Omega^*;
\end{array}
\right.
\]
\item Under Assumption II,
\[
\bar{u}_\varepsilon(y)= \left\{
\begin{array}{cl}
\displaystyle\int^{y}_{c}(G(t)-D_\varepsilon(c))/E_\varepsilon^{-1}((G(t)-D_\varepsilon(c))^2)dt,&
y\in[c,q_\varepsilon^*(c)]\subset\Omega^*,\\
\\
0,& \text{elsewhere}\ \text{in}\ \Omega^*,
\end{array}
\right.
\]
\end{itemize}
where $E_\varepsilon$ and $G$ are defined as
\[
\left\{
\begin{array}{ll}
E_\varepsilon(\gamma):=\displaystyle \gamma^2\ln({\rm
e}^{\alpha^2}\gamma^{2\varepsilon}),&
\gamma\in[{\rm e}^{-\alpha^2/(2\varepsilon)},1],\\
\\
G(y):=\displaystyle y^2/2, & y\in[p_\varepsilon^*(d),d]\ \text{or}\
[c,q_\varepsilon^*(c)].
\end{array}
\right.
\]
$E_\varepsilon^{-1}$ stands for the inverse of $E_\varepsilon$,
$C_\varepsilon(d)$ and $p^*_\varepsilon(d)$ are constants depending
on $d$ and $\varepsilon$, while $D_\varepsilon(c)$ and
$q^*_\varepsilon(c)$ are constants depending on $c$ and
$\varepsilon$.
\end{thm}
By letting $\varepsilon\to 0^+$, one can solve the optimization
problems for the expectation of the real-valued variable
$Y\in\Omega^*$.
\begin{thm}
For the maximization problem ($\mathcal{P}^{(1)}$)(or the
minimization problem ($\mathcal{P}^{(2)}$)), there exists a global
maximizing(or minimizing) distribution density $f^-$ satisfying
(2)-(5).
%
\end{thm}
Furthermore, under Assumption III, one is able to deal with the
Monge transfer problem $(1)$ with the optimal distribution
densities. In the following, we construct a sequence of mass
transfer mappings ${\bf s}_{\varepsilon}$ approximating an optimal
mapping ${\bf s}^*$. Let
$$F(x):=\int_a^xf^+(t)dt,\ x\in[a,b].$$ If $f^+>0$, then $F$ is
monotonously increasing with respect to $x\in[a,b]$, invertible and
its inverse is denoted as
$$F^{-1}:[0,1]\to[a,b].$$ In addition, let
$$Q_{\varepsilon}(y):=\int_{p_\varepsilon^*(d)}^y\bar{u}_\varepsilon(t)dt,\ y\in[p_\varepsilon^*(d),d].$$ Since $\bar{u}_\varepsilon>0$ in $(p_\varepsilon^*(d),d)$, then $Q_{\varepsilon}$ is
monotonously increasing with respect to
$y\in[p_\varepsilon^*(d),d]$, invertible and its inverse is denoted
as
$$Q_{\varepsilon}^{-1}:[0,1]\to[p_\varepsilon^*(d),d].$$
Furthermore, let
$$R_{\varepsilon}(y):=\int_{c}^y\bar{u}_\varepsilon(t)dt,\ y\in[c,q_\varepsilon^*(c)].$$ Since $\bar{u}_\varepsilon>0$ in $(c,q_\varepsilon^*(c))$, then $R_{\varepsilon}$ is
monotonously increasing with respect to
$y\in[c,q_\varepsilon^*(c)]$, invertible and its inverse is denoted
as
$$R_{\varepsilon}^{-1}:[0,1]\to[c,q_\varepsilon^*(c)].$$
\begin{thm}
Assume that $f^+(x)>0$, $x\in[a,b]$. For the Monge transfer problem
(1) under Assumption I and Assumption III, there exists a sequence
of strictly increasing(or decreasing) mappings represented
explicitly as
\[
{\bf s}_\varepsilon(x)=\begin{array}{cl} \displaystyle
Q_{\varepsilon}^{-1}(F(x)),&x\in[a,b];
\end{array}
\]
or
\[
{\bf s}_\varepsilon(x)=\begin{array}{cl} \displaystyle
Q_{\varepsilon}^{-1}(1-F(x)),&x\in[a,b].
\end{array}
\]
While for the Monge transfer problem (1) under Assumption II and
Assumption III, there exists a sequence of strictly increasing(or
decreasing) mappings represented explicitly as
\[
{\bf s}_\varepsilon(x)=\begin{array}{cl} \displaystyle
R_{\varepsilon}^{-1}(F(x)),&x\in[a,b];
\end{array}
\]
or
\[
{\bf s}_\varepsilon(x)=\begin{array}{cl} \displaystyle
R_{\varepsilon}^{-1}(1-F(x)),&x\in[a,b].
\end{array}
\]
\end{thm}

The rest of the paper is organized as follows. In Section 2, first
we introduce some useful notations which will simplify the proof
considerably. Then we apply the canonical dual transformation to
deduce a sequence of perfect dual problems
($\mathcal{P}^{(\varepsilon)}_d$), corresponding to
$(\mathcal{P}^{(\varepsilon)})$ and a pure complementary energy
principle. Next we apply the canonical duality theory to prove
Theorem 1.1, 1.2 and 1.3. A few remarks will conclude our discussion.\\

\section{Proof of the main results}
\subsection{Useful notations}
Before proving the main results, first and foremost, we introduce
some useful
notations.\\

\begin{itemize}
\item $\theta_\varepsilon$ is the corresponding G\^{a}teaux derivative of
$H^{(\varepsilon)}$ with respect to $u_{\varepsilon,y}$ given by
$$\theta_\varepsilon(y)={\rm e}^{(u_{\varepsilon,y}^2-\alpha^2)/(2\varepsilon)}u_{\varepsilon,y}.$$
\item $\Phi^{(\varepsilon)}$ is a nonlinear
geometric mapping given by
$$\Phi^{(\varepsilon)}(u_\varepsilon):=(u_{\varepsilon,y}^2-\alpha^2)/(2\varepsilon).$$
For convenience's sake, denote
$$\xi_\varepsilon:=\Phi^{(\varepsilon)}(u_\varepsilon).$$ It is evident that $\xi_\varepsilon$ belongs to the
function space $\mathscr{U}$ given by
$$\mathscr{U}:=
\Big\{\phi\ \Big|\ \phi\leq 0\Big\}.$$
\item $\Psi^{(\varepsilon)}$ is a canonical energy
defined as $$\Psi^{(\varepsilon)}(\xi_\varepsilon):=\varepsilon{\rm
e}^{\xi_\varepsilon},$$ which is a convex function with respect to
$\xi_\varepsilon$.
\item $\zeta_\varepsilon$ is the corresponding G\^{a}teaux derivative of
$\Psi^{(\varepsilon)}$ with respect to $\xi_\varepsilon$ given by
$$\zeta_\varepsilon=\varepsilon{\rm e}^{\xi_\varepsilon},$$ which is invertible with respect to
$\xi_\varepsilon$ and belongs to the function space
$\mathscr{V}^{(\varepsilon)}$,
$$\mathscr{V}^{(\varepsilon)}:=\Big\{\phi\ \Big|\ 0<\phi\leq \varepsilon\Big\}.$$
\item
$\Psi^{(\varepsilon)}_\ast$ is defined as
\[
\Psi^{(\varepsilon)}_\ast(\zeta_\varepsilon):=\xi_\varepsilon\zeta_\varepsilon-\Psi^{(\varepsilon)}(\xi_\varepsilon)=\zeta_\varepsilon(\ln(\zeta_\varepsilon/\varepsilon)-1).
\]
\item $\lambda_\varepsilon$ is defined as $$\lambda_\varepsilon:=\zeta_\varepsilon/\varepsilon,$$ and belongs
to the function space $\mathscr{V}$,
$$\mathscr{V}:=\Big\{\phi\ \Big|\ 0<\phi\leq 1\Big\}.$$
\end{itemize}
\subsection{Canonical duality techniques}
\begin{defi}
By Legendre transformation, one defines a Gao-Strang total
complementary energy functional $\Xi^{(\varepsilon)}$,
\[
\Xi^{(\varepsilon)}(u_\varepsilon,\zeta_\varepsilon):=\displaystyle\int_{U^{(\varepsilon)}}\Big\{\Phi^{(\varepsilon)}(u_\varepsilon)\zeta_\varepsilon-\Psi^{(\varepsilon)}_\ast(\zeta_\varepsilon)
-|y|u_\varepsilon\Big\}dy.
\]
\end{defi}
Next we introduce an important {\it criticality criterium} for
$\Xi^{(\varepsilon)}$.
\begin{defi}
$(\bar{u}_\varepsilon, \bar{\zeta}_\varepsilon)$ is called a
critical pair of $\Xi^{(\varepsilon)}$ if and only if \beq
D_{u_\varepsilon}\Xi^{(\varepsilon)}(\bar{u}_\varepsilon,\bar{\zeta}_\varepsilon)=0,
\eneq and \beq
D_{\zeta_\varepsilon}\Xi^{(\varepsilon)}(\bar{u}_\varepsilon,\bar{\zeta}_\varepsilon)=0,
\eneq where $D_{u_\varepsilon}, D_{\zeta_\varepsilon}$ denote the
partial G\^ateaux derivatives of $\Xi^{(\varepsilon)}$,
respectively.
\end{defi} Indeed, by variational calculus, we have the following
observation from (11) and (12).
\begin{lem}
On the one hand, for any fixed
$\zeta_\varepsilon\in\mathscr{V}^{(\varepsilon)}$, $(11)$ is
equivalent to the equilibrium equation
\[
\begin{array}{ll}\displaystyle (\lambda_\varepsilon \bar{u}_{\varepsilon,y})_y+|y|=0,& \
\text{\rm in}\ U^{(\varepsilon)}.\end{array}
\]
On the other hand, for any fixed $u_\varepsilon$ satisfying (2)-(5),
(12) is consistent with the constructive law
\[
\Phi^{(\varepsilon)}(u_\varepsilon)=D_{\zeta_\varepsilon}\Psi^{(\varepsilon)}_\ast(\bar{\zeta}_\varepsilon).
\]
\end{lem}
Lemma 2.3 indicates that $\bar{u}_\varepsilon$ from the critical
pair $(\bar{u}_\varepsilon,\bar{\zeta}_\varepsilon)$ solves the
Euler-Lagrange equation (10).
\begin{defi}
From Definition 2.1, one defines the Gao-Strang pure complementary
energy $I^{(\varepsilon)}_d$ in the form
\[
I^{(\varepsilon)}_d[\zeta_\varepsilon]:=\Xi^{(\varepsilon)}(\bar{u}_\varepsilon,\zeta_\varepsilon),
\]
where $\bar{u}_\varepsilon$ solves the Euler-Lagrange equation (10).
\end{defi}
For convenience's sake, we give another representation of the pure
energy $I^{(\varepsilon)}_d$ by the following lemma.
\begin{lem} The
pure complementary energy functional $I^{(\varepsilon)}_d$ can be
rewritten as
\[
I^{(\varepsilon)}_d[\zeta_\varepsilon]=-1/2\int_{U^{(\varepsilon)}}\Big\{{\varepsilon\theta_\varepsilon^2/\zeta_\varepsilon}+\alpha^2\zeta_\varepsilon/\varepsilon+2\zeta_\varepsilon(\ln(\zeta_\varepsilon/\varepsilon)-1)\Big\}dy,
\]
where $\theta_\varepsilon$ satisfies \beq
\theta_{\varepsilon,y}+|y|=0\ \text{in}\ U^{(\varepsilon)}, \eneq
equipped with a hidden boundary condition.
\end{lem}
\begin{proof}
Through integrating by parts, one has
\[
\begin{array}{lll}
I^{(\varepsilon)}_d[\zeta_\varepsilon]&=&\displaystyle-\underbrace{\int_{U^{(\varepsilon)}}\Big\{(\zeta_\varepsilon
\bar{u}_{\varepsilon,y}/\varepsilon)_y+|y|\Big\}\bar{u}_\varepsilon dy}_{(I)}\\
\\
&&-\underbrace{1/2\int_{U^{(\varepsilon)}}\Big\{\zeta_\varepsilon\bar{u}_{\varepsilon,y}^2/\varepsilon+\alpha^2\zeta_\varepsilon/\varepsilon+2\zeta_\varepsilon(\ln(\zeta_\varepsilon/\varepsilon)-1)\Big\}dy.}_{(II)}\\
\\
\end{array}
\]
Since $\bar{u}_\varepsilon$ solves the Euler-Lagrange equation (10),
then the first part $(I)$ disappears. Keeping in mind the definition
of $\theta_\varepsilon$ and $\zeta_\varepsilon$, one reaches the
conclusion.
\end{proof}
With the above discussion, next we establish a sequence of dual
variational problems to the approximation problems
($\mathcal{P}^{(\varepsilon)}$).
\begin{equation}
(\mathcal{P}_d^{(\varepsilon)}):\displaystyle\max_{\zeta_\varepsilon\in\mathscr{V}^{(\varepsilon)}}\Big\{I^{(\varepsilon)}_d[\zeta_\varepsilon]=-1/2\int_{U^{(\varepsilon)}}\Big\{{\varepsilon\theta_\varepsilon^2/\zeta_\varepsilon}+\alpha^2\zeta_\varepsilon/\varepsilon+2\zeta_\varepsilon(\ln(\zeta_\varepsilon/\varepsilon)-1)\Big\}dy\Big\}.
\end{equation}
Indeed, by calculating the G\^{a}teaux derivative of
$I_d^{(\varepsilon)}$ with respect to $\zeta_\varepsilon$, one has
\begin{lem} The variation of $I_d^{(\varepsilon)}$ with respect to $\zeta_\varepsilon$
leads to the dual algebraic equation (DAE), namely, \beq
\theta_\varepsilon^2={\bar{\zeta}_\varepsilon}^2(2\ln(\bar{\zeta}_\varepsilon/\varepsilon)+\alpha^2/\varepsilon)/\varepsilon,
\eneq where $\bar{\zeta}_\varepsilon$ is from the critical pair
$(\bar{u}_\varepsilon,\bar{\zeta}_\varepsilon)$.
\end{lem}
Taking into account the notation of $\lambda_\varepsilon$,  the
identity (15) can be rewritten as \beq
\theta_\varepsilon^2=E_\varepsilon(\lambda_\varepsilon)={\lambda}_\varepsilon^2\ln({\rm
e}^{\alpha^2}{\lambda}_\varepsilon^{2\varepsilon}). \eneq It is
evident $E_\varepsilon$ is monotonously increasing with respect to
$\lambda_\varepsilon\in[{\rm e}^{-\alpha^2/(2\varepsilon)},1]$. As a
matter of fact, $\theta_\varepsilon^2$ has the following asymptotic
expansion by using Taylor's expansion formula for
$\ln\lambda_\varepsilon$ at the point 1.
\begin{lem}
When $\varepsilon$ is sufficiently small, $\theta_\varepsilon^2$ has
the asymptotic expansion,
$$\theta_\varepsilon^2=(\alpha^2-2\varepsilon)\lambda_\varepsilon^2+2\varepsilon\lambda_\varepsilon^3+R_\varepsilon(\lambda_\varepsilon),$$
where the remainder term
\[
|R_\varepsilon(\lambda_\varepsilon)|\leq \varepsilon
\] uniformly for any $\lambda_\varepsilon\in[{\rm e}^{-\alpha^2/(2\varepsilon)},1]$.
\end{lem}
\subsection{Proof of Theorem 1.1}
From the above discussion, one deduces that, once
$\theta_\varepsilon$ is given, then the analytic solution of the
Euler-Lagrange equation (10) can be represented as \beq
\bar{u}_\varepsilon(y)=\displaystyle\int^{y}_{y_0}\eta_\varepsilon(t)dt,
\eneq where $y\in U^{(\varepsilon)}, y_0\in\partial
U^{(\varepsilon)}$,
$\eta_\varepsilon:=\theta_\varepsilon/\lambda_\varepsilon$. In the
following, we will determine the support $U^{(\varepsilon)}$. First
and foremost, we prove several useful lemmas.
\begin{lem}
For $\forall\ \varepsilon>0$, \begin{itemize}
\item Under Assumption I, $\forall\ s\in[c,d)$, there exists a
unique solution $\bar{u}_{\varepsilon}\in C^\infty[s,d]$ of the
Euler-Lagrange equation (10) with Dirichlet boundary in the form of
(17). \item Under Assumption II, for $\forall\ t\in(c,d]$, there
exists a unique solution $\bar{u}_{\varepsilon}\in C^\infty[c,t]$ of
the Euler-Lagrange equation (10) with Dirichlet boundary in the form
of (17).
\end{itemize}
\end{lem}
\begin{proof}
{\it First case:}\\

In $[s,d]$, one has a general solution for the differential equation
$\theta_{\varepsilon,y}=-y$ in the form of
\[
\theta_\varepsilon(y)=-G(y)+C_\varepsilon=-y^2/2+C_\varepsilon,\ \
y\in[s,d]\subset(0,+\infty).
\]
From the identity (16), one sees that there exists a unique
$C^\infty$ function $\lambda_\varepsilon\in[{\rm
e}^{-\alpha^2/(2\varepsilon)},1]$. By paying attention to the
Dirichlet boundary $\bar{u}_\varepsilon(s)=0$, one has the analytic
solution $\bar{u}_\varepsilon$ in the following form,
\[
\bar{u}_\varepsilon(y)=\int^{y}_{s}\eta_\varepsilon(x)dx,\ \ \ \
y\in[s,d].
\]
Recall that
\[
\bar{u}_\varepsilon(d)=\int^{G^{-1}(C_\varepsilon)}_{s}\eta_\varepsilon(x)dx+\int^{d}_{G^{-1}(C_\varepsilon)}\eta_\varepsilon(x)dx=0,
\]
and one can determine the constant $C_\varepsilon\in(s^2/2,d^2/2)$
uniquely. Indeed, let
\[
\mu_\varepsilon(y,r):=(-G(y)+r)/\lambda_\varepsilon(y,r)
\]
and
\[
M_{\varepsilon}(r):=\int^{G^{-1}(r)}_{s}\mu_\varepsilon(y,r)dy+\int^{d}_{G^{-1}(r)}\mu_\varepsilon(y,r)dy,
\]
where $\lambda_\varepsilon(y,r)$ is from (16). It is evident that
$\lambda_\varepsilon$ depends on $C_\varepsilon$. As a matter of
fact, $M_\varepsilon$ is strictly increasing with respect to
$r\in(s^2/2,d^2/2)$, which leads to
\[
C_\varepsilon=M_{\varepsilon}^{-1}(0).
\]
In fact, $C_\varepsilon$ depends on $s$ and the contradiction method
shows that $C_\varepsilon$ is strictly increasing with respect to
$s\in[c,d)$.\\

{\it Second case:}\\

In $[c,t]$, one has a general solution for the differential equation
$\theta_{\varepsilon,y}=y$ in the form of
\[
\theta_\varepsilon(y)=G(y)-D_\varepsilon=y^2/2-D_\varepsilon,\ \
y\in[c,t]\subset(-\infty,0).
\]
From the identity (16), one sees that there exists a unique
$C^\infty$ function $\lambda_\varepsilon\in[{\rm
e}^{-\alpha^2/(2\varepsilon)},1]$. By paying attention to the
Dirichlet boundary $\bar{u}_\varepsilon(c)=0$, one has the analytic
solution $\bar{u}_\varepsilon$ in the following form,
\[
\bar{u}_\varepsilon(y)=\int^{y}_{c}\eta_\varepsilon(x)dx,\ \ \ \
y\in[c,t].
\]
Recall that
\[
\bar{u}_\varepsilon(t)=\int^{G^{-1}(D_\varepsilon)}_{c}\eta_\varepsilon(x)dx+\int^{t}_{G^{-1}(D_\varepsilon)}\eta_\varepsilon(x)dx=0,
\]
and one can determine the constant $D_\varepsilon\in(t^2/2,c^2/2)$
uniquely. Indeed, let
\[
\mu_\varepsilon(y,r):=(G(y)-r)/\lambda_\varepsilon(y,r)
\]
and
\[
N_{\varepsilon}(r):=\int^{G^{-1}(r)}_{c}\mu_\varepsilon(y,r)dy+\int^{t}_{G^{-1}(r)}\mu_\varepsilon(y,r)dy,
\]
where $\lambda_\varepsilon(y,r)$ is from (16). It is evident that
$\lambda_\varepsilon$ depends on $D_\varepsilon$. As a matter of
fact, $N_\varepsilon$ is strictly decreasing with respect to
$r\in(t^2/2,c^2/2)$, which leads to
\[
D_\varepsilon=N_{\varepsilon}^{-1}(0).
\]
In fact, $D_\varepsilon$ depends on $t$ and the contradiction method
shows that $D_\varepsilon$ is strictly decreasing with respect to
$t\in(c,d]$.\\

\end{proof}
\begin{lem}
For $\forall\ \varepsilon>0$,
\begin{itemize}
\item Under Assumption I, if $d-c$ is sufficiently large such
that
\beq\displaystyle\int_c^d\int_d^y\Big(-x^2/2+C_\varepsilon\Big)/\Big(\lambda_{\varepsilon}(x,C_\varepsilon)\Big)dxdy>1\
\text{\rm when\ $d\in$\ Supp}\bar{u}_{\varepsilon}, \eneq then there
exists a unique $p_\varepsilon^*(d)$ such that
\[
\text{\rm Supp}\bar{u}_{\varepsilon}=[p_\varepsilon^*(d),d]\
\text{and}\ \int_{p_\varepsilon^*(d)}^d\bar{u}_{\varepsilon}(y)dy=1.
\]
\item Under Assumption II, if $d-c$ is sufficiently large such
that
\beq\displaystyle\int_c^d\int_c^y\Big(x^2/2-D_\varepsilon\Big)/\Big(\lambda_{\varepsilon}(x,D_\varepsilon)\Big)dxdy>1\
\text{\rm when\ $c\in$\ Supp}\bar{u}_{\varepsilon},\eneq then there
exists a unique $q_\varepsilon^*(c)$ such that
\[
\text{\rm Supp}\bar{u}_{\varepsilon}=[c,q_{\varepsilon}^{*}(c)]\
\text{and}\ \int^{q_\varepsilon*(c)}_c\bar{u}_{\varepsilon}(y)dy=1.
\]
\end{itemize}
\end{lem}
\begin{proof}

{\it First Part:}\\

Let $\text{\rm Supp}\bar{u}_{\varepsilon}=[s,d]$ and define a
function $\Pi:[c,d)\to\mathbb{R}^+$ as follows,
\[\Pi(s):=\displaystyle\int_s^d\int_d^y\Big(-x^2/2+C_\varepsilon(s)\Big)/\Big(\lambda_{\varepsilon}(x,C_\varepsilon(s))\Big)dxdy.\]
Indeed, since $C_\varepsilon$ is strictly increasing with respect to
$s\in[c,d)$, as a result, it is easy to check that $\Pi$ is a
strictly decreasing function with respect to $s\in[c,d)$. The first
assertion follows immediately when we recall (18).\\

{\it Second Part:}\\

Let $\text{\rm Supp}\bar{u}_{\varepsilon}=[c,t]$ and define a
function $\Pi:(c,d]\to\mathbb{R}^+$ as follows,
\[\Pi(t):=\displaystyle\int_c^t\int_c^y\Big(x^2/2-D_\varepsilon(t)\Big)/\Big(\lambda_{\varepsilon}(x,D_\varepsilon(t))\Big)dxdy.\]
Indeed, since $D_\varepsilon$ is strictly decreasing with respect to
$t\in[c,d)$, as a result, it is easy to check that $\Pi$ is a
strictly increasing function with respect to $t\in(c,d]$. The second
assertion follows immediately when we recall (19).
\end{proof}

Next we verify that $\bar{u}_\varepsilon$ is exactly a global
minimizer for ($\mathcal{P}^{(\varepsilon)}$) and
$\bar{\zeta}_\varepsilon$ is a global maximizer for
($\mathcal{P}^{(\varepsilon)}_d$).
\begin{lem}(Canonical Duality Theory)
For $\forall\ \varepsilon>0$, $\bar{u}_\varepsilon$ in Lemma 2.9 is
a global minimizer for the approximation problem
($\mathcal{P}^{(\varepsilon)}$). And the corresponding
$\bar{\zeta}_\varepsilon$ is a global maximizer for the dual problem
($\mathcal{P}_d^{(\varepsilon)}$). Moreover, the following duality
identity holds, \beq
I^{(\varepsilon)}[\bar{u}_\varepsilon]=\displaystyle\min_{u_\varepsilon}I^{(\varepsilon)}[u_\varepsilon]=\Xi^{(\varepsilon)}(\bar{u}_\varepsilon,\bar{\zeta}_\varepsilon)=\displaystyle\max_{\zeta_\varepsilon}I_d^{(\varepsilon)}[\zeta_\varepsilon]=I_d^{(\varepsilon)}[\bar{\zeta}_\varepsilon],
\eneq where $u_\varepsilon$ is subject to the constraints (2)-(5)
and $\zeta_\varepsilon\in\mathscr{V}^{(\varepsilon)}$.
\end{lem}
Lemma 2.10 demonstrates that the maximization of the pure
complementary energy functional $I_d^{(\varepsilon)}$ is perfectly
dual to the minimization of the potential energy functional
$I^{(\varepsilon)}$. In effect, the identity (20) indicates there is
no duality gap between them.
\begin{proof}
On the one hand, for any function $\phi\in
W_0^{1,\infty}(U^{(\varepsilon)})$, the second variational form
$\delta_\phi^2I^{(\varepsilon)}$ is equal to\beq
\int_{U^{(\varepsilon)}}{\rm
e}^{(\bar{u}_{\varepsilon,x}^2-\alpha^2)/(2\varepsilon)}\Big\{(\bar{u}_{\varepsilon,x}
\phi_x)^2/\varepsilon+\phi_x^2\Big\}dx.\eneq On the other hand, for
any function $\psi\in\mathscr{V}^{(\varepsilon)}$, the second
variational form $\delta_\psi^2I_d^{(\varepsilon)}$ is equal to
\beq-\int_{U^{(\varepsilon)}}\Big\{\varepsilon\theta_\varepsilon^2\psi^2/\bar{\zeta}_\varepsilon^3+\psi^2/\bar{\zeta}_\varepsilon\Big\}dx.
\eneq From (21) and (22), one deduces immediately that
\[
\delta^2_\phi I^{(\varepsilon)}[\bar{u}_\varepsilon]\geq0,\ \
\delta_\psi^2I_d^{(\varepsilon)}[\bar{\zeta}_\varepsilon]\leq0.
\]
\end{proof}
Consequently, we reach the conclusion of Theorem 1.1 by summarizing
the above discussion.
\subsection{Proof of Theorem 1.2}
According to Rellich-Kondrachov Compactness Theorem, since
$$\displaystyle\sup_{\varepsilon}|\bar{u}_\varepsilon|\leq \alpha(d-c)$$ and
$$\displaystyle\sup_{\varepsilon}|\bar{u}_{\varepsilon,y}|\leq\alpha,$$ then, there exists a
subsequence $\{\bar{u}_{\varepsilon_k}\}_{\varepsilon_k}$ and
$f^{-}\in W_0^{1,\infty}(\Omega^*)\cap C(\overline{\Omega^*})$ such
that \beq\bar{u}_{\varepsilon_k}\rightarrow f^-\ (k\to\infty)\
\text{in}\ L^\infty(\Omega^*),\eneq
\beq\bar{u}_{\varepsilon_k,y}\ \overrightharpoon{*}\ f^-_{y}\
(k\to\infty)\ \text{weakly\ $\ast$\ in}\ L^\infty(\Omega^*).\eneq It
remains to check that $f^-$ satisfies (2)-(5). From (23), one knows
\beq\bar{u}_{\varepsilon_k}\rightarrow f^-\ (k\to\infty)\
\text{a.e.\ in}\ \Omega^*.\eneq According to Lebesgue's dominated
convergence theorem,
\[\displaystyle\int_{\Omega^*}f^-(y)dy=\lim_{k\to\infty}\int_{\Omega^*}\bar{u}_{\varepsilon_k}(y)dy=1.\]
From (24), one has
\[\|f^-_y\|_{L^\infty(\Omega^*)}\leq\displaystyle\liminf_{k\to\infty}\|\bar{u}_{\varepsilon_k,y}\|_{L^\infty(\Omega^*)}\leq\sup_{k\to\infty}\|\bar{u}_{\varepsilon_k,y}\|_{L^\infty(\Omega^*)}\leq\alpha.\]
\subsection{Proof of Theorem 1.3}
On the one hand, under Assumption I and III, one solves the
differential problems
\[\bar{u}_\varepsilon({\bf s}_\varepsilon){\rm d}{\bf s}_\varepsilon=f^+(x){\rm d}x,\ \ {\bf s}_\varepsilon(b)=d,\]
or
\[\bar{u}_\varepsilon({\bf s}_\varepsilon){\rm d}{\bf s}_\varepsilon=-f^+(x){\rm d}x,\ \ {\bf s}_\varepsilon(a)=d\]
respectively. On the other hand, under Assumption II and III, one
solves the differential problems
\[\bar{u}_\varepsilon({\bf s}_\varepsilon){\rm d}{\bf s}_\varepsilon=f^+(x){\rm d}x,\ \ {\bf s}_\varepsilon(a)=c,\]
or
\[\bar{u}_\varepsilon({\bf s}_\varepsilon){\rm d}{\bf s}_\varepsilon=-f^+(x){\rm d}x,\ \ {\bf s}_\varepsilon(b)=c\]
respectively. Our conclusion follows immediately.\\
\\

{\bf Concluding Remarks}:\\
\\
In this paper, we mainly focus on the construction of 1-1 mappings
approximating the optimal Monge transfer mapping. Rather than
numerical simulation, we give the explicit representation of the
approximating mappings by applying the canonical duality method.
Together with the other convex approximation mechanisms used by L.
A. Caffarelli, W. Gangbo, R. J. McCann, N. S. Trudinger, L. C.
Evans, W. Gangbo and X. J. Wang
\cite{Ca2,DM,Evans1,Evans2,GC1,GC2,Wang1,Wang2}, we provide another
viewpoint, namely, nonlinear differential equation approach, for the Monge transfer problem.\\
\\
As a matter of fact, in a similar manner, Assumption I and
Assumption II can be relaxed to the whole $\mathbb{R}$ as long as
$\Omega\bigcap\Omega^*=\emptyset$. It remains to discuss various
cases when $\Omega\bigcap\Omega^*\neq\emptyset$, in which case,
several new assumptions will have to be introduced. Furthermore,
optimal transfer mapping in the $n$-dimensional case will be given
in
a sequential paper. The canonical duality method proves to be useful and can also be applied in the discussion of $p$-Laplacian problems and optimal probability density for $p$-th moment etc. \cite{LU2,LU3}.\\
\\

{\bf Acknowledgment}: The main results in this paper were obtained
during a research collaboration at the Federation University
Australia in August, 2016. Both authors wish to thank Professor
David Y. Gao for his hospitality and financial support. This project
is partially supported by US Air Force Office of Scientific Research
(AFOSR FA9550-10-1-0487), Natural Science Foundation of Jiangsu
Province (BK 20130598), National Natural Science Foundation of China
(NSFC 71273048, 71473036, 11471072), the Scientific Research
Foundation for the Returned Overseas Chinese Scholars, Fundamental
Research Funds for the Central Universities on the Field Research of
Commercialization of Marriage between China and Vietnam (No.
2014B15214). This work is also supported by Open Research Fund
Program of Jiangsu Key Laboratory of Engineering Mechanics,
Southeast University (LEM16B06). In particular, the authors also
express their deep gratitude to the referees for their careful
reading and useful remarks.

\end{document}